\documentclass[journal]{IEEEtran}
\usepackage{amsfonts,amsmath,graphicx,cite,bm,amssymb,amsthm,enumerate,epsfig,psfrag,cases,mathtools}
\renewcommand{\(}{\left(}
\renewcommand{\)}{\right)}
\renewcommand{\[}{\left[}
\renewcommand{\]}{\right]}

\renewcommand{\a}{\mathbf{a}}

\newcommand{\n}{\mathbf{n}}

\renewcommand{\S}{\mathbf{S}}

\newcommand{\z}{\mathbf{z}}

\newcommand{\x}{\mathbf{x}}

\newcommand{\I}{\mathbf{I}}

\newcommand{\A}{\mathbf{A}}
\newcommand{\T}{\mathbf{T}}
\newcommand{\U}{\mathbf{U}}
\newcommand{\M}{\mathbf{M}}
\newcommand{\N}{\mathbb{N}}

\renewcommand{\u}{\mathbf{u}}
\renewcommand{\v}{\mathbf{v}}

\newcommand{\B}{\mathbf{B}}

\newcommand{\Tr}[1]{{\rm{Tr}}\left(#1\right)}

\newcommand{\End}[1]{{\rm{End}}}

\renewcommand{\log}[1]{{\rm{log}}#1}
\renewcommand{\arg}[1]{{\rm{arg}}#1}

\renewcommand{\vec}[1]{{\rm{vec}}\(#1\)}

\newtheorem{lemma}{Lemma}
\newtheorem{definition}{Definition}
\newtheorem{theorem}{Theorem}

\newtheorem{corollary}{Corollary}

\newcommand{\norm}[1]{\left\lVert#1\right\rVert}

\usepackage{fontenc}
\usepackage{inputenc}
\usepackage[square,sort,compress,comma,numbers]{natbib}
\usepackage[table]{xcolor}
\usepackage[ruled,vlined]{algorithm2e}
\usepackage{epstopdf}

\begin{document}
\title{Performance Analysis of Tyler's Covariance Estimator}

\author{Ilya Soloveychik and Ami Wiesel, \\
 Rachel and Selim Benin School of Computer Science and Engineering, The Hebrew University of Jerusalem, Israel

\thanks{This work was partially supported by the Intel Collaboration Research Institute for Computational Intelligence and Kaete Klausner Scholarship.}
}

\maketitle

\begin{abstract}
This paper analyzes the performance of Tyler's M-estimator of the scatter matrix in elliptical populations. We focus on the non-asymptotic setting and derive the estimation error bounds depending on the number of samples $n$ and the dimension $p$. We show that under quite mild conditions the squared Frobenius norm of the error of the inverse estimator decays like $p^2/n$ with high probability.
\end{abstract}


\section{Introduction}
Estimation of large covariance matrices, particularly in situations where the data dimension $p$ and the sample size $n$ are of close magnitudes has recently attracted considerable attention. Estimators in this field can be classified based on the underlying distribution and the additional structure assumptions. Most of the research is traditionally devoted to the multivariate Gaussian setting which is currently well understood. Various algorithms based on different structures and their performance analysis have been derived for the Gaussian distributions, see e.g. \cite{ledoit2003improved,bickel2008regularized,rothman2008sparse}. This allows for higher reliability and better tuning of regularization parameters. Recently, similar challenges have appeared in the more ambitious setting of non-Gaussian and robust estimation. A prominent approach in this area is Tyler's scatter estimator \cite{tyler1987distribution}. The goal of this paper is to analyze its non-asymptotic behavior as defined below more rigorously. We start by first reviewing the state of the art in both Gaussian and Tyler's covariance estimation, and then introduce our result.

Most of the works on covariance estimation address the Gaussian scenario. When the number of samples is greater than the dimension, the Maximum Likelihood Estimator (MLE) of the covariance exists with probability one and coincides with the Sample Covariance Matrix (SCM). Recently, there has emerged a great amount of literature on regularized versions of this estimator, their parameter tuning and performance analysis, such as shrinkage-estimator, see e.g. \cite{ledoit2003improved,rothman2008sparse}.

Roughly speaking, covariance estimation performance analyses can be divided into three regimes. The first is classical asymptotic analysis. This approach assumes that the dimension of the underlying sample space $p$ is fixed and that the number of samples grows $n\rightarrow \infty$. Typical results in this regime concern asymptotic consistency and computation of asymptotic variance of an estimator in comparison to the Cramer-Rao lower bound. The second regime is also asymptotic and is based on Random Matrix Theory (RMT). It treats the case where both $n\rightarrow \infty$ and $p\rightarrow \infty$ while their ratio tends to a fixed number $p/n \rightarrow \zeta \in [0;\infty)$. One of the most important results of this kind concerning SCM is the Marchenko-Pastur law \cite{marvcenko1967distribution} defining the asymptotic distribution of the eigenvalues of the properly defined sequence of SCM. Other typical results include contributions on the bulk and edge regions of the spectrum of SCM, see e.g. \cite{tracy1996orthogonal, el2007tracy, bai2008limit, baik2006eigenvalues}. These works led to numerous theoretical and practical breakthroughs in covariance estimation and its applications \cite{cardoso2008cooperative, couillet2011fluctuations,mestre2008modified,schmidt1986multiple, couillet2011eigen, vallet2012improved}.

The third kind of results stems from the non-asymptotic analysis of the SCM, see e.g. \cite{vershynin2012compressed, vershynin2012close,srivastava2011covariance}. In contrast to the previously described regimes, the non-asymptotic results answer the question what error bound can one obtain given the values of $n$ and $p$. They are usually formulated in the form of concentration of measure results:
\begin{equation*}
\norm{\widetilde{\bm\Theta} - \bm\Theta_0} \leq f(n,p) \text{ with probability at least } 1-e^{g(n,p)},
\label{concentr_in}
\end{equation*}
where the particular estimator $\widetilde{\bm\Theta}$ and norm should be specified and the functions $f$ and $g$ should satisfy some properties. Unlike the asymptotic regimes, such kinds of results describe the speed of convergence of the estimator to the true covariance matrix $\bm\Theta_0$.

Recently non-asymptotic analysis of SCM in Gaussian distributions has become popular among signal processing society due to advances in high dimensional statistics. In particular, \cite{bickel2008regularized,ravikumar2011high, rothman2008sparse} consider regularized covariance estimation and include its performance analysis. A common thread to all of these works is that the estimators are defined as the solutions of convex optimization problems, and the analysis is directly related to the notion of strong convexity.

In many applications the underlying multivariate distribution is actually non-Gaussian and robust covariance estimation  methods are required. This occurs whenever the probability distribution of the measurements is heavy-tailed or a small proportion of the samples represents outlier behavior, \cite{huber1964robust, maronna1976robust}. A common robust estimator of scatter is due to Tyler \cite{tyler1987distribution}. Given $n$ independent, identically distributed (i.i.d.) measurements $\x_i \in \mathbb{R}^p,i=1,\dots,n,$ Tyler's shape matrix estimator is defined as the solution to the fixed point equation
\begin{equation}
\T = \frac{p}{n}\sum_{i=1}^n \frac{\x_i\x_i^T}{\x_i^T\T^{-1}\x_i}.
\label{tylerequ}
\end{equation}
When $\x_i$ are Generalized Elliptically (GE) distributed \cite{frahm2004generalized}, their shape matrix $\bm\Theta_0$ is positive definite and $n>p$, Tyler's estimator exists with probability one and is a consistent estimator of $\bm\Theta_0$ up to a positive scaling factor. The GE family includes as particular cases generalized Gaussian distribution, compound Gaussian, elliptical and many others \cite{frahm2004generalized}. Therefore, it has been successfully used to replace the SCM in many applications such as anomaly detection in wireless sensor networks \cite{chen2011robust}, antenna array processing \cite{ollila2003robust} and radar detection \cite{abramovich2007time, ollila2012complex, bandiera2010knowledge, pascal2008covariance}.


Performance analysis of Tyler's estimator dates back to the original work in classical robust statistics literature \cite{tyler1987distribution}. In the classical asymptotic regime where $p$ is fixed and $n \rightarrow \infty$ it was shown that, when properly scaled, Tyler's estimator is strongly consistent with the true covariance matrix, if it exists, and is asymptotically normally distributed. Its asymptotic variance, which coincides with the Cramer-Rao bound, was analyzed in \cite{besson2013fisher, greco2013cramer}. RMT regime performance results were also reported. For the case $n, p \rightarrow \infty, p/n \rightarrow 0$ it was shown in \cite{duembgen1997asymptotic} that the condition number of Tyler's estimator multiplied by the inverse true shape matrix of the underlying distribution converges to $1 + O(\sqrt{p/n})$. In \cite{frahm2012semicircle} it was demonstrated that the empirical spectral distribution of $\sqrt{n/p}(\T-\I)$ converges to the semicircle law in probability, when the underlying population is white and Tyler's estimator is properly scaled.

In the paper \cite{zhang2014marchenko} the authors proved that Tyler's M-estimator converges in operator norm to the SCM matrix as $n, p \rightarrow \infty$ and $p/n \rightarrow \zeta \in (0,1)$, when data samples follow the standard normal distribution. The authors extended this result to elliptical distributions and proved that the empirical spectral density of Tyler's M-estimator converges to the Marchenko-Pastur distribution. The paper \cite{zhang2014marchenko} also quantified the non-asymptotic behavior of Tyler's M-estimator, but in a different way with emphases on the RMT regime, the Marchenko-Pastur law, and comparison to the SCM. The paper \cite{couillet2014large} analyzed the asymptotic behavior of regularized Tyler's estimator in the RMT regime, which allowed optimal parameter tuning. Additional RMT results on other M-estimators were developed in \cite{couillet2013random, couillet2012robust}.

We focus on the the non-asymptotic analysis of Tyler's estimator for moderate values of $n$ and $p$ based on the concentration of measure phenomenon. Note that the estimator is not given in closed form and has to be iteratively computed using the fixed point iteration (\ref{tylerequ}). In order to exploit the optimization based machinery discussed above, we rely on an alternative derivation of Tyler's estimator. In particular, the estimator can also be obtained as an MLE of normalized GE distributed vectors defined as $\frac{\x_i}{\norm{\x_i}}$, \cite{frahm2004generalized}. Our method is therefore based on the analysis of the negative log-likelihood function of this distribution. We present high probability error bounds on the deviation of a properly scaled Tyler's estimator from the true underlying shape matrix. In particular, we prove that as long as $n$ is larger than $p$ the Frobenius norm of the error in inverse matrices decays like $\frac{p}{\sqrt{n}}$ with high probability. We also show that our performance bounds exhibit correct asymptotic behavior in the RMT regime. This contribution basically complements the previous classical asymptotic and RMT analysis of Tyler's method. The derivation generally follows the optimization based approach due to \cite{rothman2008sparse, ravikumar2011high}.

The paper is organized as following: first we introduce notations, state the problem, the main result and provide a discussion of it. Then we outline the strategy of the proof and bring in a few auxiliary results. After this we prove the main theorem. The body of the article contains the sketch of the proof with the statements of the most significant lemmas. Finally, we provide numerical simulations illustrating the obtained results. The proofs of lemmas and the rest of the auxiliary statements are left for the Appendices.

\subsection{Notations}
Denote by $\mathcal{S}(p)$ the linear space of $p \times p$ symmetric real matrices and by $\mathcal{P}(p) \subset \mathcal{S}(p)$ the closed cone of positive semi-definite matrices. $\I$ stands for the identity matrix of a proper dimension.

We endow $\mathcal{S}(p)$ with the scalar product $(\A,\B) = \Tr{\A\B}$, which induces the Frobenius norm on it. $\norm{\cdot}$ will denote the Euclidean norm for vectors, $\norm{\cdot}_F$ - the Frobenius norm and $\norm{\cdot}_2$ - the spectral norm for matrices. The linear space $\mathbb{R}^p$ is treated as a column vector space with the standard inner product. Given a matrix $\A$ we denote by $\vec{\A}$ a column vector obtained by stacking the columns of $\A$.

For a matrix $\A \in \mathcal{P}(p)$ denote by $\lambda_{\min}(\A)$ and $\lambda_{\max}(\A)$ its minimal and maximal eigenvalues, correspondingly. When $\lambda_{\min}(\A) > 0$ we write $\A \succ 0$ and denote by $\kappa(\A) = \frac{\lambda_{\max}(\A)}{\lambda_{\min}(\A)}$ its condition number. $|\A|$ stands for the determinant of $\A$.

Let $Q$ be a quadratic form over a finite dimensional Euclidean space $\mathcal{V}$, define its minimal and maximal eigenvalues as
\begin{equation}
\lambda_{\min}(Q) = \inf_{\a \in \mathcal{V}, \norm{\a}=1} Q(\a),\quad\lambda_{\max}(Q) = \sup_{\a \in \mathcal{V}, \norm{\a}=1} Q(\a).
\label{lambda_def}
\end{equation}
The norm of $Q$ is defined as
\begin{equation*}
\norm{Q}_2 = \max(|\lambda_{\min}(Q)|, |\lambda_{\max}(Q)|).
\end{equation*}

For $n$ instances $a_1,\dots,a_n$ of scalars, vectors, matrices or functions we denote by $\widehat{a}$ their arithmetic average, when the index of summation is obvious from the context.

Matrices are denoted by Capital bold letters $\M$, column vectors by non-capital bold $\v$, scalars by non-capital $r$, operators and quadratic forms by Capital $T$ letters.

\section{Tyler's estimator as a MLE}
We define Tyler's estimator as an MLE of a shape matrix parameter of a specific real spherical $p$-dimensional distribution. The likelihood function of this distribution is later used to derive error bounds of the estimator using its curvature properties.
\begin{definition}
\label{def}
Assume $\bm\Theta_0 \in \mathcal{P}(p), \bm\Theta_0 \succ 0$, then the function
\begin{equation}
p(\x) = \frac{\Gamma(p/2)}{2\sqrt{\pi^p}} \frac{1}{\sqrt{|\bm\Theta_0|}(\x^T  \bm\Theta_0^{-1}\x)^{p/2}}
\label{tyler_distr}
\end{equation}
is a probability density function of a vector $\x \in \mathbb{R}^p$ lying on a unit sphere. This distribution is usually referred to as the Angular Central Gaussian (ACG) distribution on a sphere \cite{tyler1987statistical} and we denote it as $\x \sim \mathcal{U}(\bm\Theta_0)$. The matrix $\bm\Theta_0$ is referred to as a shape matrix of the distribution and is a multiple of the covariance matrix of $\x$.
\end{definition}

The ACG distribution is closely related to the class of GE distributions, which includes Gaussian, compound Gaussian, elliptical, skew-elliptical, ACG and other distributions, \cite{frahm2007tyler}. An important property of the GE family is that the shape matrix of a population does not change when the vector is divided by its Euclidean norm \cite{frahm2004generalized, frahm2007tyler}. After normalization, any GE vector becomes ACG distributed. This allows us to treat all these distributions together using a single robust estimator.

Assuming $\bm\Theta \in \mathcal{P}(p), \bm\Theta \succ 0$ and given $n>p$ i.i.d. copies of a vector $\x\sim \mathcal{U}(\bm\Theta)\colon\x_i,1=1,\dots,n$ we derive the MLE estimator of the shape matrix. For this sake introduce the scaled negative log-likelihood function:
\begin{equation}
\widetilde{f}(\bm\Theta;\x) = \log|\bm\Theta| + p\log(\x^T  \bm\Theta^{-1} \x).
\label{neg_ll}
\end{equation}
The function (\ref{neg_ll}) is non-convex in $\bm\Theta$. Nevertheless, its critical point given as the solution to (\ref{tylerequ}) provides the global minima with probability one, \cite{wiesel2012geodesic, wiesel2012unified, zhang2013multivariate}.

The negative log-likelihood (\ref{neg_ll}) is invariant under multiplication of the shape matrix by a positive constant, thus we are only interested in the estimation of the shape matrix up to a positive scalar factor. In order to obtain a unique MLE we fix the scale of the estimator by assuming that $\Tr{\bm\Theta_0^{-1}}$ of the true covariance matrix is known (or arbitrarily fixed). Specifically, we define Tyler's estimator to be the solution to the program
\begin{equation}
\T = \arg \left\{
\begin{aligned}
& \underset{\bm\Theta}{\min}
& & \frac{1}{n}\sum_{i=1}^n \widetilde{f}(\bm\Theta;\x_i) \\
& \text{subject to}
& & \Tr{\bm\Theta^{-1}} = \Tr{\bm\Theta_0^{-1}} \label{add_def_tyler}.
\end{aligned}
\right.
\end{equation}

\section{the Main Result}
In this section we introduce the main result of the paper and compare it with the similar results and related works. Denote $\bm\Omega_0=\bm\Theta_0^{-1}, \underline \lambda:=\lambda_{\min}(\bm\Theta_0) = \lambda_{\max}^{-1}(\bm\Omega_0) > 0$ and
\begin{equation}
\cos{\phi_0} = \frac{\Tr{\bm\Omega_0}}{\norm{\I}_F\norm{\bm\Omega_0}_F}>0.
\label{phi_def_sin}
\end{equation}

\begin{theorem}
\label{thm:main_res}
Assume we are given $n > p$ i.i.d. copies of $\x \sim \mathcal{U}(\bm\Theta_0)$, then for $\theta\geq 0$ with probability at least
\begin{align}
&1-2\exp\(\frac{-\theta^2}{2(1+1.7\frac{\theta}{\sqrt{n}})}\)\nonumber\\
&-2p^2\exp\(-\frac{n\cos^2{\phi_0}}{77\ln(7p)(1+\frac{2}{p})}\)\(1+\frac{15\cdot10^3(1+\frac{2}{p})^4}{n^2\cos^8{\phi_0}}\)
\label{thm_b}
\end{align}
Tyler's estimator (\ref{add_def_tyler}) satisfies
\begin{equation}
\norm{\T^{-1} - \bm\Theta_0^{-1}}_F \leq \theta\frac{10}{\underline\lambda \cos^2{\phi_0}}\frac{p+2}{\sqrt{n}}.
\label{tyler_bound}
\end{equation}
\end{theorem}

We note that for the Gaussian populations the same technique provides a similar up to a constant factor bound for the SCM estimator, suggesting that both estimators are bounded by a multiple of $\frac{p}{\sqrt{n}}$ with high probability.

\subsection{Identity Covariance}
The value of $\cos{\phi_0}$ is close to $1$ if the condition number $\kappa(\bm\Theta_0)$ is close to $1$ and gets close to $0$ if the matrix $\bm\Theta_0$ has many small eigenvalues and few large ones. As we know the estimation of the matrix becomes less stable in the latter case, as the theorem suggests. To make the statement of Theorem \ref{thm:main_res} easier to grasp let us treat the case of $\bm\Theta_0 = \I$ in the following
\begin{corollary}
\label{cor:main_res}
Assume we are given $n > p$ i.i.d. copies of $\x \sim \mathcal{U}(\I), \theta < \frac{\sqrt{n}}{4}$, then
\begin{align*}
&\mathbb{P}\(\norm{\T^{-1} - \I}_F \geq 10\theta\frac{p+2}{\sqrt{n}}\) \leq 2\exp\(\frac{-\theta^2}{3}\) \\
&+2p^2\exp\(-\frac{n}{77\ln(7p)(1+\frac{2}{p})}\)\(1+\frac{15\cdot10^3(1+\frac{2}{p})^4}{n^2}\).
\end{align*}
\end{corollary}
This corollary illustrates the essential behavior of the bound: as $n$ and $p$ get large the second probability term vanishes and we get a large deviation-type bound.

\subsection{Choice of the Norm}
Our result and the corresponding results of \cite{rothman2008sparse, ravikumar2011high} for the Gaussian settings are formulated in terms of the inverse of the estimated matrix. Actually the necessity to formulate the convergence rates in terms of inverse matrices is twofold. First of all, in many applications one is mostly interested in the inverse covariance matrix estimation, e.g. in regression and denoising. Secondly, this choice is also dictated by the properties of the negative log-likelihood function. It is more convenient for the analysis to parametrize this function by the inverse covariance matrix.

In addition, the fact that the Frobenius and not the spectral norm is used in the bounds also deserves explanation. In the Gaussian case stronger bounds can be obtained with the spectral norm. In fact it can be shown that the rate of convergence of the SCM matrix to the true covariance is of order $\sqrt{\frac{p}{n}}$ in the spectral norm, see \cite{vershynin2012compressed} and references therein. In our case the problem of obtaining spectral norm non-asymptotic bounds is much more involved and remains an open question.

\section{The Proof Strategy}
\label{reason}
The proof of Theorem \ref{thm:main_res} follows a technique similar to that of \cite{rothman2008sparse, bradley2012sample, ravikumar2011high}. The method of obtaining sample complexity rates proposed by these papers consists of finding the smallest ball around the true inverse covariance that contains the estimator with high probability. In its turn, this is done by considering the second order Taylor expansion of the sample average negative log-likelihood and estimating the maximal radius of the ball, for which the increase of this function on the boundary of the ball is positive with high probability.

The reasoning is as following: assume we are given a continuous function $g$ over a finite dimensional Euclidean space which is known to have a unique local minimum. If we also know that $g(\x) < g|_{\partial\mathcal{B}(\x;\rho)}$, where $\mathcal{B}(\x;\rho)$ is a closed ball of radius $\rho$ around $\x$, then the unique minumum of the function belongs to the open ball $\mathcal{B}(\x;\rho) \backslash \partial\mathcal{B}(\x;\rho)$. Indeed, $g$ is continuous over a compact set $\mathcal{B}(\x;\rho)$, thus reaches its extrema on it. The minimum of $g$ cannot lie on the boundary since $g|_{\partial\mathcal{B}(\x;\rho)} > g(\x)$, so it is strictly inside the ball.

Strong convexity of $g$ is a sufficient condition for it having a unique local minimum. For this purpose any notion of convexity is suitable, since the uniqueness of the minimum does not depend on the metric. The second condition $g(\x) < g|_{\partial\mathcal{B}(\x;\rho)}$ is usually demonstrated using strong local convexity of $g$ in the vicinity of the true parameter value. Due to the scale invariance, the Hessian of Tyler's objective has a zero eigenvalue in the vicinity of the true parameter. The trace constraint addresses this invariance and ensures that the Hessian is positive definite.

\section{Preliminary Results for the ACG Distribution}
The analysis becomes easier if the negative log-likelihood is parametrized by the inverse shape matrix. We denote $\bm\Omega = \bm\Theta^{-1}$ and write
\begin{equation}
f(\bm\Omega;\x) = \widetilde{f}(\bm\Omega^{-1};\x) = -\log|\bm\Omega| + p\log(\x^T  \bm\Omega \x).
\label{log_lik_n}
\end{equation}
Slightly abusing the notations, below we refer to (\ref{log_lik_n}) as the negative log-likelihood of the ACG population.

\subsection{Derivatives and Their Expectations}
The negative log-likelihood gradient and Hessian read as
\begin{equation*}
\nabla f_{\bm\Omega} = - \bm\Omega^{-1} + p\frac{\x \x^T}{\x^T  \bm\Omega \x},
\end{equation*}
\begin{equation}
\nabla^2 f_{\bm\Omega} = \bm\Omega^{-1} \otimes \bm\Omega^{-1} - p\frac{\x \x^T}{\x^T \bm\Omega \x} \otimes \frac{\x \x^T}{\x^T \bm\Omega \x}.
\label{raw_hes}
\end{equation}
These and their expectations can be considered as linear and quadratic forms over $\mathcal{S}(p)$ respectively:
\begin{equation}
\nabla f_{\bm\Omega}(\U) = - \Tr{\bm\Omega^{-1}\U} + p\frac{\x^T \U \x}{\x^T  \bm\Omega \x},
\label{der_log_lik}
\end{equation}
\begin{align}
&\nabla^2 f_{\bm\Omega}(\U) =  \Tr{\bm\Omega^{-1}\U \bm\Omega^{-1}\U} - p\(\frac{\x^T \U \x}{\x^T  \bm\Omega \x}\)^2 \nonumber\\
&= \norm{\bm\Omega^{-1/2}\U\bm\Omega^{-1/2}}_F^2 - p\(\frac{\x^T \U \x}{\x^T  \bm\Omega \x}\)^2.
\label{hes_log_lik}
\end{align}
Let us compute the expectations of the derivatives for $\x \sim \mathcal{U}(\bm\Theta_0)$. Denote
\begin{equation*}
T_{\bm\Omega} = \mathbb{E}\(\nabla f_{\bm\Omega}\), H_{\bm\Omega} = \mathbb{E}\( \nabla^2 f_{\bm\Omega}\),
\end{equation*}
\begin{equation}
R^\nu(\U,\bm\Omega;\bm\Theta_0) = \mathbb{E}\[\(\frac{\x^T\U \x}{\x^T\bm\Omega \x}\)^\nu\], \nu \in \N,
\label{r_def}
\end{equation}
then
\begin{equation*}
T_{\bm\Omega}(\U) = - \Tr{\bm\Omega^{-1}\U} + p R^1(\U,\bm\Omega;\bm\Theta_0),
\end{equation*}
\begin{equation}
H_{\bm\Omega}(\U) =  \Tr{(\bm\Omega^{-1}\U)^2} - p R^2(\U,\bm\Omega;\bm\Theta_0).
\label{exp_hes}
\end{equation}
In particular, at the true parameter value $\bm\Omega_0$ we obtain
\begin{equation*}
T_{\bm\Omega_0}(\U) = 0, \forall \U \in \mathcal{S}(p),
\end{equation*}
\begin{align*}
&H_{\bm\Omega_0}(\U) = \frac{p\Tr{(\bm\Theta_0\U)^2} - \(\Tr{\bm\Theta_0\U}\)^2}{p+2} \\
&= \frac{p\norm{\bm\Theta_0^{1/2}\U\bm\Theta_0^{1/2}}_F^2 - \(\Tr{\bm\Theta_0^{1/2}\U\bm\Theta_0^{1/2}}\)^2}{p+2}.
\end{align*}

These formulas follow directly from (\ref{der_log_lik}), (\ref{hes_log_lik}) and formulas (\ref{fir_r_mom}), (\ref{sec_r_mom}) from Appendix \ref{f_app}.  Our analysis is based on strong convexity of the Hessian. It can be easily shown that $H_{\bm\Omega_0}(\U)$ has a one-dimensional kernel spanned by $\bm\Omega_0$. In what follows we show that the trace constraint in (\ref{add_def_tyler}) eliminates this direction and makes the Hessian strongly convex in the vicinity of $\bm\Omega_0$ with high probability.

In order to explore the convexity properties of $H_{\bm\Omega_0}(\U)$ let us state the following simple
\begin{lemma}
\label{pyt_lem}
Let $\mathcal{V}$ be a Euclidean space and $\mathcal{S} \subset \mathcal{V}$ be its subspace of codimension one with normal vector $\n$, then for any $\v \in \mathcal{V}$
\begin{equation*}
\sin^2(\angle \u,\v) \geq \cos^2(\angle\n,\v) = \frac{(\n,\v)}{\norm{\n}\norm{\v}}, \forall \u \in \mathcal{S}.
\end{equation*}
\end{lemma}
\begin{proof}
Among all the vectors $\u\in \mathcal{S}$ the one that minimizes the angle $(\angle \u,\v)$ is coplanar with $\n$ and $\v$. Now the statement follows from Pythagoras theorem.
\end{proof}
Let us now turn to the space $\mathcal{S}(p)$ and consider its subspace $\mathcal{L}$ defined by the condition $\Tr{\U} = 0$. Apply to this subspace the following linear transformation:
\begin{equation}
\mathcal{L'} = \bm\Theta_0^{1/2}\mathcal{L}\bm\Theta_0^{1/2}.
\label{l_map}
\end{equation}
$\bm\Omega_0$ is a normal vector of $\mathcal{L'}$, since $(\bm\Theta_0^{1/2}\U\bm\Theta_0^{1/2},\bm\Omega_0) = \Tr{\U} = 0, \forall \U \in \mathcal{L}$. Recall that
\begin{align*}
&\Tr{\bm\Theta_0^{1/2}\U\bm\Theta_0^{1/2}} = (\bm\Theta_0^{1/2}\U\bm\Theta_0^{1/2},\I) \\ &=\sqrt{p}\norm{\bm\Theta_0^{1/2}\U\bm\Theta_0^{1/2}}_F \cos(\angle\bm\Theta_0^{1/2}\U\bm\Theta_0^{1/2},\I),
\end{align*}
and set $\n = \bm\Omega_0, \v = \I$ to apply Lemma \ref{pyt_lem} and get
\begin{align*}
&H_{\bm\Omega_0}(\U) = \frac{p}{p+2}\sin^2(\angle\bm\Theta_0^{1/2}\U\bm\Theta_0^{1/2},\I)\norm{\bm\Theta_0^{1/2}\U\bm\Theta_0^{1/2}}_F^2 \\
&\geq \frac{p}{p+2}\cos^2(\angle\I,\bm\Omega_0)\norm{\bm\Theta_0^{1/2}\U\bm\Theta_0^{1/2}}_F^2, \forall \U \in \mathcal{L}.
\end{align*}
For brevity denote
\begin{equation*}
\cos{\phi_0} = \cos(\angle\I,\bm\Omega_0) = \frac{\Tr{\bm\Omega_0}}{\norm{\I}_F\norm{\bm\Omega_0}_F} = \frac{\Tr{\bm\Omega_0}}{\sqrt{p}\norm{\bm\Omega_0}_F} \geq \frac{1}{\kappa(\bm\Theta_0)},
\end{equation*}
as was already defined in (\ref{phi_def_sin}). This quantity is closely related to the notion of sphericity, \cite{ledoit2002some}. We see that $\cos{\phi_0}$ defines the convexity properties of $H_{\bm\Omega_0}(\U)$ restricted to $\mathcal{L}$ and thus plays a crucial role in the bound provided by Theorem \ref{thm:main_res}.

\section{Proof of the Main Theorem}
Denote
\begin{equation*}
\widehat{f_{\bm\Omega}} = \frac{1}{n}\sum_{i=1}^n f(\bm\Omega;\x_i),
\end{equation*}
and let us use the Taylor polynomial formula with remainder in the Lagrange form. The expansion in the vicinity of the true inverse covariance matrix $\bm\Omega_0$ reads as:
\begin{equation}
\widehat{f_{\bm\Omega}} - \widehat{f_{\bm\Omega_0}} = \nabla\widehat{f_{\bm\Omega_0}}(\Delta\bm\Omega) +\frac{1}{2}\nabla^2 \widehat{f_{\overline{\bm\Omega}}}(\Delta\bm\Omega),
\label{taylor_expansion}
\end{equation}
where $\bm\Omega = \bm\Omega_0 + \Delta\bm\Omega, \overline{\bm\Omega} = \bm\Omega_0 + \alpha \Delta\bm\Omega, \alpha \in [0,1]$ and $\Delta\bm\Omega \in \mathcal{L}$. The sample average negative log-likelihood function (\ref{taylor_expansion}) is zero for $\Delta\bm\Omega = 0$ and we want to show that with high probability it is positive on the boundary of some ball.

Note that in spite of the fact the negative log-likelihood (\ref{log_lik_n}) is not globally convex in $\bm\Omega$, it is geodesically convex as shown in \cite{wiesel2012geodesic} and thus has a unique global minimizer no matter what metric over $\mathcal{P}(p)$ is considered. This justifies application of the technique described in section \ref{reason}.

In order to make the proof more transparent and instructive, and the notations more concise let us note that the distributions of the functions in both sides of (\ref{taylor_expansion}) are invariant under the linear map $'$ defined in (\ref{l_map})
\begin{equation*}
\mathcal{L'} = \bm\Theta_0^{1/2}\mathcal{L}\bm\Theta_0^{1/2},
\end{equation*}
applied to matrices $\bm\Omega_0, \Delta\bm\Omega$ followed by a change of random vectors from $\x_i \sim \mathcal{U}(\bm\Theta_0)$ to $\x_i' \sim \mathcal{U}(\I)$ for $i=1,\dots,n$. This linear transformation relocates the domain under consideration from the vicinity of $\bm\Omega_0$ into a vicinity of $\I$ and (\ref{taylor_expansion}) reads as
\begin{equation*}
\widehat{f_{\bm\Omega'}} - \widehat{f_{\I}} = \nabla\widehat{f_{\I}}(\Delta\bm\Omega') +\frac{1}{2}\nabla^2 \widehat{f_{\overline{\bm\Omega'}}}(\Delta\bm\Omega'),
\end{equation*}
where $\bm\Omega' = \I + \Delta\bm\Omega', \overline{\bm\Omega'} = \I + \alpha \Delta\bm\Omega', \Delta\bm\Omega' \in \mathcal{L'}, \alpha \in [0,1]$. Since $'$ is a tension-compression map, the maximal and minimal distance changes are known and we can use an inequality
\begin{equation}
\underline \lambda \norm{\T^{-1} - \bm\Theta_0^{-1}}_F \leq \norm{\T'^{-1} - \bm\Theta_0'^{-1}}_F
\label{lamb_b}
\end{equation}
to establish high probability bounds for inverse Tyler's estimator given the respective bounds on its $'$-image.

\subsection{Gradient and Hessian Bounds}
\begin{lemma} \label{lem:lem_1}
Uniformly over $\U' \in \mathcal{L'}$
\begin{equation*}
\mathbb{P}\(\left|\nabla \widehat{f_{\I}}(\U')\right| \geq t p\norm{\U'}_F\) \leq 2\exp\(\frac{-nt^2}{2(1+1.7t)}\), \forall t \geq 0.
\end{equation*}
\end{lemma}
\begin{proof}
Is provided in Appendix \ref{conc_bounds}.
\end{proof}

We proceed by establishing concentration properties of the sample mean Hessian restricted to $\mathcal{L'}$ at the point $\bm\Omega'$.
\begin{lemma}\label{lem:lem_2}
Let $\bm\Omega' = \I + \Delta \bm\Omega',\: \Delta \bm\Omega' \in \mathcal{L'}$ and $\varepsilon = \norm{\Delta \bm\Omega'}_2 < 1$, then for $\frac{1}{2\sqrt{2}}\frac{1+1/p^2}{\ln(32\sqrt{2}p^2)} < \eta < \frac{1}{2\sqrt{2}}$ uniformly over $\U' \in \mathcal{L'}$
\begin{align*}
&\mathbb{P}\(\nabla^2 \widehat{f_{\bm\Omega'}}(\U') \leq H_{\bm\Omega'}(\U') - \eta\frac{\sqrt{2}\norm{\U'}_F^2}{(1-\varepsilon)^2} \) \\
& \leq 2p^2\exp\(-\frac{n\eta\sqrt{2}}{8\ln(32\sqrt{2}p^2)}\)\(1+\frac{7}{n^2\eta^4}\).
\end{align*}
\end{lemma}
\begin{proof}
Is provided in Appendix \ref{conc_bounds}.
\end{proof}

This lemma shows that with probability depending on the parameters of the problem the sample average Hessian is not far from the expected Hessian at the same point. The next result shows the expected Hessian at a point $\overline{\bm\Omega'}$ close to the true parameter is lower bounded by a fraction of the expected Hessian at the true parameter value.
\begin{lemma} \label{lem:lem_3}
For $\overline{\bm\Omega'} = \I + \alpha \Delta\bm\Omega', \alpha \in [0,1],\: \Delta\bm\Omega' \in \mathcal{L'}, \: \varepsilon=\norm{\Delta\bm\Omega'}_2 \leq \frac{p}{6(p+2)}\cos^2{\phi_0}$ the expected Hessian $H_{\overline{\bm\Omega'}}(\Delta\bm\Omega')$ is bounded from below by
\begin{equation*}
H_{\overline{\bm\Omega'}}(\Delta\bm\Omega') \geq \frac{p}{2(p+2)}\cos^2{\phi_0}\norm{\Delta\bm\Omega'}_F^2.
\end{equation*}
\end{lemma}
\begin{proof}
Is provided in Appendix \ref{conc_bounds}.
\end{proof}

\begin{corollary}
\label{cor_lem_3}
Under the conditions of Lemma \ref{lem:lem_3} for $\tau$ satisfying
$\frac{1+1/p^2}{\ln(32\sqrt{2}p^2)} \leq \tau\frac{p(1-\varepsilon)^2\cos^2{\phi_0}}{p+2} \leq 1$
\small
\begin{align*}
&\mathbb{P}\(\frac{1}{2}\nabla^2 \widehat{f_{\overline{\bm\Omega'}}}(\Delta\bm\Omega') \leq (1-\tau)\frac{p}{4(p+2)}\cos^2{\phi_0}\norm{\Delta\bm\Omega'}_F^2\) \nonumber\\
&\leq 2p^2\exp\(-\frac{n\tau\cos^2{\phi_0}}{46\ln(7p)(1+\frac{2}{p})}\)\(1+\frac{2\cdot10^3(1+\frac{2}{p})^4}{n^2\tau^4\cos^8{\phi_0}}\).
\end{align*}
\normalsize
\end{corollary}
\begin{proof}
Is provided in Appendix \ref{conc_bounds}.
\end{proof}

It now follows that with probability at least
\small
\begin{align}
&1-2\exp\(\frac{-nt^2}{2(1+1.7t)}\) \nonumber\\
&- 2p^2\exp\(-\frac{n\tau\cos^2{\phi_0}}{46\ln(7p)(1+\frac{2}{p})}\)\(1+\frac{2\cdot10^3(1+\frac{2}{p})^4}{n^2\tau^4\cos^8{\phi_0}}\),
\label{true_prob}
\end{align}
\begin{equation*}
\widehat{f_{\bm\Omega'}} - \widehat{f_{\I}} \geq - t p \norm{\Delta\bm\Omega'}_F + (1-\tau)\frac{p}{4(p+2)}\cos^2{\phi_0}\norm{\Delta\bm\Omega'}_F^2.
\end{equation*}
\normalsize
Demand positivity of the right-hand side to get
\begin{equation}
\norm{\Delta\bm\Omega'}_F > \frac{4t}{1-\tau}\frac{p+2}{\cos^2{\phi_0}}.
\label{true_b}
\end{equation}
Set $\tau = \frac{3}{5}, t = \frac{\theta}{\sqrt{n}}$ and use (\ref{lamb_b}) to get the statement of the theorem.

\section{Numerical Results}
In this section we provide numerical simulations supporting our analysis. Figure \ref{perf_n} compares the behavior of Tyler's estimator with its $0.95$ and $0.5$-probability bounds for $p=50, \bm\Theta_0 = \I$. For a given $n$ Tyler's estimator error bound was obtained by minimizing (\ref{true_b}) with respect to $t$ and $\tau$ under the probability constraint given by equating (\ref{true_prob}) to $0.95$ or $0.5$ respectively.

Figure \ref{perf_p} verifies the dependence of the performance on the dimension.
\begin{figure}
\centering
\includegraphics[height=3.1in]{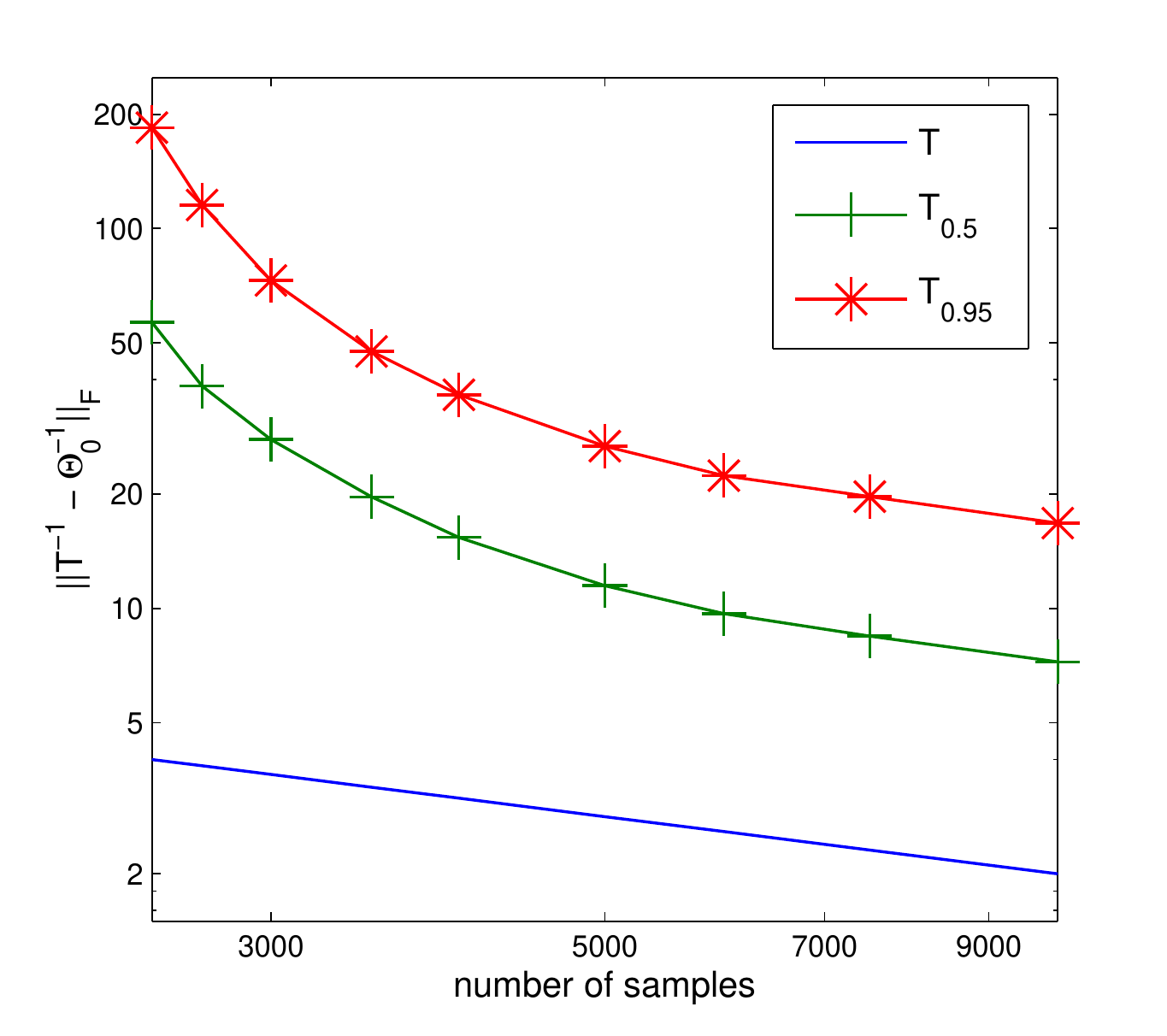}
\caption{Tyler's estimator performance bounds, $p=50$, $\bm\Theta_0 = \I$.}
\label{perf_n}
\end{figure}
\begin{figure}
\centering
\includegraphics[height=3.1in]{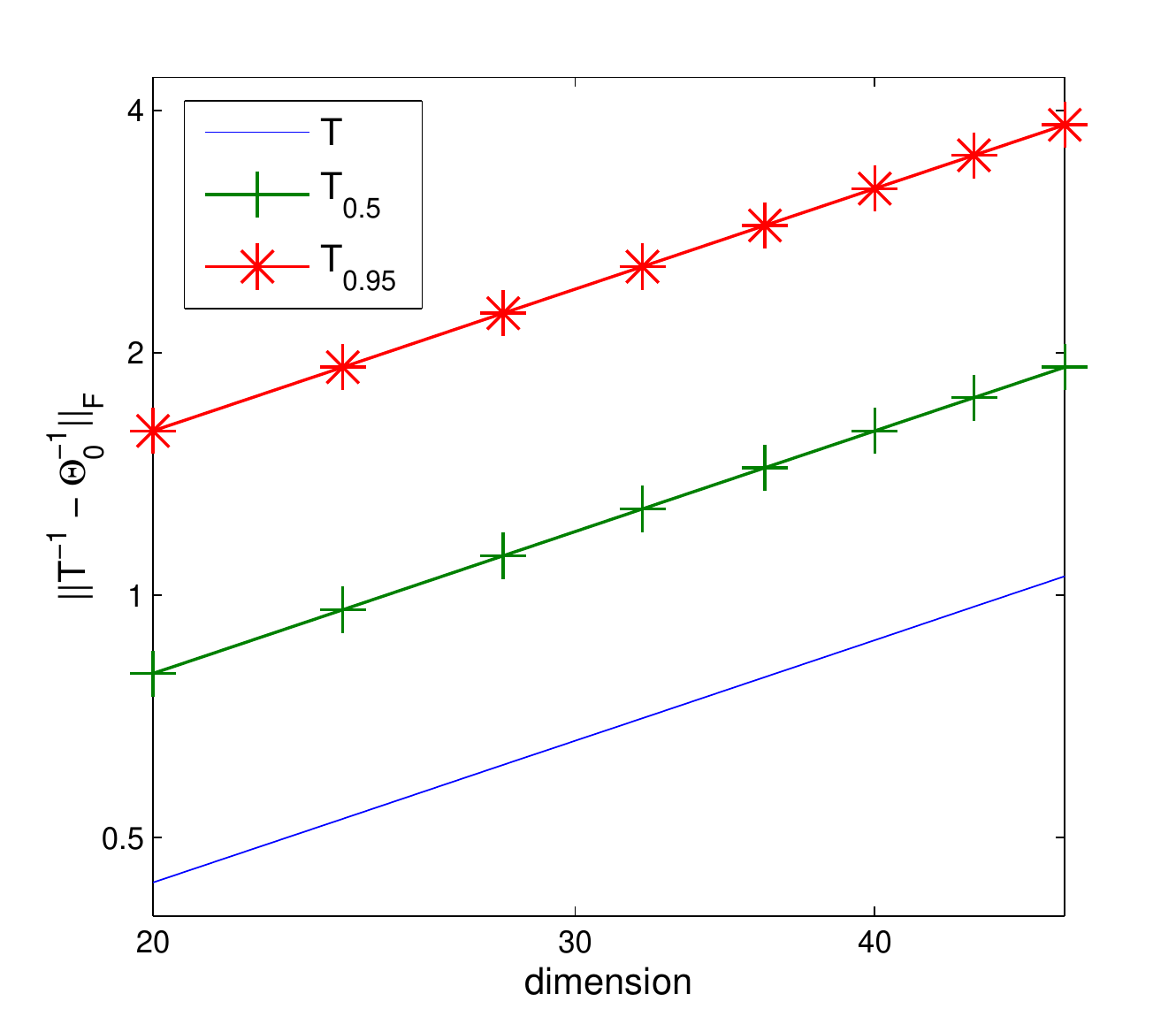}
\caption{Tyler's estimator performance bounds, $n=2500, \bm\Theta_0(p) = \I_p$.}
\label{perf_p}
\end{figure}

\section{Acknowledgment}
The authors would like to thank the associate editor and anonymous reviewers whose input greatly improved the paper.

\begin{appendices}
\section{Moments of Ratios of Quadratic Forms}
\label{f_app}

\begin{lemma}
\label{int_bounds}
Let $\x' \sim \mathcal{U}(\I), \U' \in \mathcal{S}(p), \bm\Omega' = \I + \Delta \bm\Omega' \succ 0$ and $\varepsilon = \norm{\Delta \bm\Omega'}_2 < 1$, then the moments defined in (\ref{r_def}) satisfy
\begin{align}
\frac{R^\nu(\U',\I;\I)}{(1+\varepsilon)^\nu} \leq R^\nu(\U',\bm\Omega';\I) \leq \frac{R^\nu(\U',\I;\I)}{(1-\varepsilon)^\nu}, \nu \in \N.
\label{lem_ineq_r}
\end{align}
In addition
\small
\begin{align*}
&\left|R^\nu(\Delta \bm\Omega',\I + \alpha\Delta \bm\Omega';\I) - R^\nu(\Delta \bm\Omega',\I;\I) + \nu \alpha R^{\nu+1}(\Delta \bm\Omega',\I;\I)\right| \nonumber\\
&\leq \frac{\nu(\nu+1)\alpha^2}{2(1-\alpha\varepsilon)^{\nu+2}}R^{\nu+2}(\Delta \bm\Omega',\I;\I).
\end{align*}
\normalsize
\end{lemma}
\begin{proof}
Consider
\begin{equation}
R^\nu(\U',\bm\Omega';\I) = \mathbb{E}\[\(\frac{\x'^T\U' \x'}{\x'^T(\I + \Delta \bm\Omega') \x'}\)^\nu\],
\label{aux_b}
\end{equation}
since
\begin{equation*}
1 + \varepsilon \geq 1 + \x'^T\Delta \bm\Omega' \x' \geq 1 - \varepsilon,
\end{equation*}
the variable under expectation in (\ref{aux_b}) can be bounded as
\small
\begin{equation*}
\(\frac{\x'^T\U' \x'}{1 + \varepsilon}\)^\nu \leq \(\x'^T\U' \x'\(1+\x'^T\Delta \bm\Omega'\)^{-1}\)^\nu \leq \(\frac{\x'^T\U' \x'}{1 - \varepsilon}\)^\nu.
\end{equation*}
\normalsize
By taking the expectations (\ref{lem_ineq_r}) follows. We continue
\small
\begin{align*}
&R^\nu(\Delta \bm\Omega',\I + \alpha\Delta \bm\Omega';\I) = \mathbb{E}\[\(\frac{\x'^T\Delta \bm\Omega' \x'}{\x'^T(\I + \alpha\Delta \bm\Omega') \x'}\)^\nu\] \\
&= \mathbb{E}\[\(\x'^T\Delta \bm\Omega' \x'\(1+\alpha\x'^T\Delta \bm\Omega'\x'\)^{-1}\)^\nu\] \\
&=\mathbb{E}\[\(\x'^T\Delta \bm\Omega' \x'\)^\nu\(1-\nu\alpha\x'^T\Delta \bm\Omega'\x' + g(\alpha\Delta \bm\Omega',\x)\)\], \nonumber
\end{align*}
\normalsize
where $g(\alpha\Delta \bm\Omega',\x')$ is bounded by
\small
\begin{equation*}
|g(\alpha\Delta \bm\Omega',\x')| \leq \nu(\nu+1)\alpha^2\frac{(\x'^T\Delta \bm\Omega'\x')^2}{2(1-\alpha\varepsilon)^{\nu+2}}.
\end{equation*}
\normalsize
We finally obtain
\small
\begin{align*}
&|R^\nu(\Delta \bm\Omega',\I + \alpha\Delta \bm\Omega';\I) - R^\nu(\Delta \bm\Omega',\I;\I) + \nu\alpha R^{\nu+1}(\Delta \bm\Omega',\I;\I)| \\
&\leq \frac{\nu(\nu+1)\alpha^2}{2(1-\alpha\varepsilon)^{\nu+2}}R^{\nu+2}(\Delta \bm\Omega',\I;\I).
\end{align*}
\normalsize
\end{proof}

Given $\x' \sim \mathcal{U}(\I)$, we can always represent it as $\x'=\frac{\z'}{\norm{\z'}}$, where $\z'$ is standard normally distributed $\z' \sim \mathcal{N}(0, \I)$, to obtain
\begin{equation*}
\frac{\x'^T\U' \x'}{\x'^T\bm\Omega' \x'} =\frac{\z'^T\U' \z'}{\z'^T\bm\Omega' \z'},
\end{equation*}
using this identity, we develop formulas for $R^\nu(\U',\I;\I)$, \cite{de1980exact}:
\footnotesize
\begin{equation}
R^1(\U',\I;\I) = \frac{\Tr{\U'}}{p},
\label{fir_r_mom}
\end{equation}
\begin{equation}
R^2(\U',\I;\I) =\frac{\(\Tr{\U'}\)^2+2\Tr{\(\U'\)^2}}{p(p+2)} =\frac{\(\Tr{\U'}\)^2+
2\norm{\U'}_F^2}{p(p+2)},
\label{sec_r_mom}
\end{equation}
\begin{align}
R^3(\U',\I;\I) =\frac{\Tr{\U'}^3+6\Tr{\U'}\Tr{\U'^2}
+8\Tr{\U'^3}}{p(p+2)(p+4)}.
\label{t_r_mom}
\end{align}
\normalsize
A general moment bound can also be obtained as
\begin{equation}
R^{\nu}(\U',\I;\I) \leq \frac{(\frac{\nu}{2})!\norm{\U'}_F^\nu}{\sqrt{p}^\nu}, \nu=2,4,\dots.
\label{r_k_bound}
\end{equation}

\section{Concentration Bounds}
\label{conc_bounds}
\begin{lemma}(Vector Bernstein Inequality) \cite{yurinskiui1976exponential}
\label{v_bernstein_th}
Let $\bm\xi_1,\dots,\bm\xi_n \in \mathbb{R}^k$  be i.i.d zero-mean random vectors and suppose there exist $\sigma, L>0$ such that
\begin{equation*}
\mathbb{E}\norm{\bm\xi_1}^r \leq \frac{r!}{2}\sigma^2 L^{r-2}, r=2,3,\dots,
\end{equation*}
then for $t\geq 0$
\begin{equation*}
\mathbb{P}\(\norm{\widehat{\bm\xi}} \geq t \sigma\) \leq
2\exp\(\frac{-nt^2}{2(1+1.7t\frac{L}{\sigma})}\).
\end{equation*}
\end{lemma}

\begin{lemma}(Matrix Bernstein Inequality) \cite{minsker2011some}
\label{m_bernstein_th}
Let $\S_1,\dots,\S_n \in \mathcal{S}(p)$ be i.i.d zero-mean random matrices and suppose there exist $\sigma, L>0$ such that
\begin{equation*}
\mathbb{E}\norm{\S_1}_2^r \leq \frac{r!}{2}\sigma^2 L^{r-2}, r=2,3,\dots,
\end{equation*}
then for $t > \frac{\sigma}{4L}\frac{1+1/p^2}{\ln{\frac{64\sqrt{2}p^2L^2}{\sigma^2}}}$
\small
\begin{align*}
&\mathbb{P}\(\lambda_{\max}(\widehat{\S}) \geq t\sigma\) \\
&\leq 2p^2 \exp\(-\frac{nt\sigma}{8L\ln{\frac{64\sqrt{2}p^2L^2}{\sigma^2}}}\)\(1+\frac{6}{n^2t^2\sigma^2\ln^2(1+\frac{t}{\sigma})}\).
\end{align*}
\normalsize
\end{lemma}

\begin{lemma}
\label{large_dev_lin}
Let $\x_i' \sim \mathcal{U}(\I) , i=1,\dots,n$ then for $t \geq 0$
\begin{equation*}
\mathbb{P}\(\norm{p\widehat{\x'\x'^T}-\I}_F \geq t p\) \leq 2\exp\(\frac{-nt^2}{2(1+1.7t)}\)
\end{equation*}
\begin{proof}
Define $n$ centered random vectors
\begin{equation*}
\bm\xi_i = \vec{\x_i'\x_i'^T-\frac{1}{p}\I} \in \mathbb{R}^{p^2}, i=1,\dots,n,
\end{equation*}
and consider powers of their norms
\small
\begin{align*}
&\norm{\bm\xi_1}^r = \[\Tr{\(\x_1'\x_1'^T-\frac{1}{p}\I\)^2}\]^{\frac{r}{2}} \\
&= \[\Tr{\(1-\frac{2}{p}\)\x_1'\x_1'^T+\frac{1}{p^2}\I}\]^{\frac{r}{2}} = \(1-\frac{1}{p}\)^{\frac{r}{2}},
\end{align*}
\normalsize
which are deterministic quantities. Set
\begin{equation*}
\sigma = L = 1,
\end{equation*}
and apply Lemma \ref{v_bernstein_th} to obtain
\begin{equation*}
\mathbb{P}\(\norm{\widehat{\bm\xi}} \geq t\) \leq 2\exp\(\frac{-nt^2}{2(1+ 1.7t)}\).
\end{equation*}
Multiply $\widehat{\bm\xi}$ by $p$ to get the statement.
\end{proof}
\end{lemma}

\begin{proof}[Proof of Lemma \ref{lem:lem_1}]
\begin{align*}
\left|\nabla \widehat{f_{\I}}(\U')\right| = \left|\Tr{\(p\widehat{\x' \x'^T} - \I\) \U'}\right|.
\end{align*}
Apply Lemma \ref{large_dev_lin} and the Cauchy-Schwartz inequality to get the statement.
\end{proof}

\begin{proof}[Proof of Lemma \ref{lem:lem_2}]
For a linear operator $L$, its restriction $\widetilde{L}$ to a linear subspace satisfies
\begin{equation*}
\norm{\widetilde{L}}_2 \leq \norm{L}_2,
\end{equation*}
thus we can apply Lemma \ref{m_bernstein_th} to bound the deviation of $\nabla^2 \widehat{f_{\bm\Omega'}}$ restricted to $\mathcal{L'}$ from its expectation $H_{\bm\Omega'}$ also restricted to $\mathcal{L'}$. Define $n$ centered random quadratic forms
\begin{align}
&S_i(\U') = \Tr{(\bm\Omega'^{-1}\U')^2} - p\(\frac{\x_i'^T \U' \x_i'}{\x_i'^T  \bm\Omega' \x_i'}\)^2 - H_{\bm\Omega'}(\U') \nonumber\\
&= p \(R^2(\U,\bm\Omega';\I) - \(\frac{\x_i'^T \U' \x_i'}{\x_i'^T  \bm\Omega' \x_i'}\)^2\),i=1,\dots,n.
\label{med_oc}
\end{align}

Bound the moments of $S_1(\U')$
\small
\begin{align*}
&|\mathbb{E}[S_1(\U')^r]| = \left|p^r\sum_{j=0}^r (-1)^j {r \choose j} \[R^{2}(\U',\bm\Omega';\I)\]^{r-j} R^{2j}(\U',\bm\Omega';\I)\right| \\
&\leq  p^r\max_j \frac{r!}{(r-j)!j!}\(\frac{\norm{\U'}_F^2}{(1-\varepsilon)^2p}\)^{r-j} \frac{j!\norm{\U'}_F^{2j}}{(1-\varepsilon)^{2j}p^j} \leq \frac{r!}{2}\frac{2}{(1-\varepsilon)^{2r}}\norm{\U'}_F^{2r},
\end{align*}
\normalsize
which allows us to set
\begin{equation*}
\frac{\sigma}{\sqrt{2}} = L = \frac{1}{(1-\varepsilon)^2}.
\end{equation*}
Use Lemma \ref{m_bernstein_th} to get for $\frac{1}{2\sqrt{2}}\frac{1+1/p^2}{\ln(32\sqrt{2}p^2)} < \eta < \frac{1}{2\sqrt{2}}$
\small
\begin{align*}
&\mathbb{P}\(\lambda_{\min}\(\nabla^2 \widehat{f_{\bm\Omega'}}|_{\mathcal{L'}}\) \leq \lambda_{\min}(H_{\bm\Omega'}|_{\mathcal{L'}}) - \eta\sigma \) \\
&= \mathbb{P}\(\lambda_{\min}\(\widehat{S}|_{\mathcal{L'}} \) \leq -\eta\sigma \) = \mathbb{P}\(\lambda_{\max}\(-\widehat{S}|_{\mathcal{L'}} \) \geq \frac{\sqrt{2}\eta}{(1-\varepsilon)^2} \) \\
&\leq 2p^2\exp\(-\frac{n\eta\sqrt{2}}{8\ln(32\sqrt{2}p^2)}\)\(1+\frac{3(1-\varepsilon)^4}{n^2\eta^2\ln^2\(1+\frac{\eta(1-\varepsilon)^2}{\sqrt{2}}\)}\) \\
&\leq 2p^2\exp\(-\frac{n\eta\sqrt{2}}{8\ln(32\sqrt{2}p^2)}\)\(1+\frac{7}{n^2\eta^4}\).
\end{align*}
\normalsize
\end{proof}

\begin{proof}[Proof of Lemma \ref{lem:lem_3}]
From (\ref{exp_hes}) we have
\begin{equation*}
H_{\overline{\bm\Omega'}}(\Delta\bm\Omega') = \Tr{\(\overline{\bm\Omega'}^{-1}\Delta\bm\Omega'\)^2} - p R^2(\Delta\bm\Omega',\overline{\bm\Omega'};\I).
\end{equation*}
Using Lemma \ref{int_bounds} we obtain
\small
\begin{align*}
&H_{\overline{\bm\Omega'}}(\Delta\bm\Omega') = \Tr{\[(\I+\alpha\Delta\bm\Omega')^{-1}\Delta\bm\Omega'\]^2} - p R^2(\Delta\bm\Omega',\I+\alpha\Delta\bm\Omega';\I) \\
&=\Tr{\Delta\bm\Omega'^2-2\alpha\Delta\bm\Omega'^3 + h(\alpha\Delta\bm\Omega')\Delta\bm\Omega'^2} \\
&- p \[R^2(\Delta\bm\Omega',\I;\I)-2\alpha R^3(\Delta\bm\Omega',\I;\I) +l(\alpha\Delta\bm\Omega')\].
\end{align*}
\normalsize
Use the Lagrange remainder form for $h(\alpha\Delta\bm\Omega')$ and apply Lemma \ref{int_bounds} to bound $l(\alpha\Delta\bm\Omega')$ and get
\small
\begin{equation*}
|h(\alpha\Delta\bm\Omega')| \leq \frac{3(\alpha\varepsilon)^2}{(1-\alpha\varepsilon)^4},\quad|l(\alpha\Delta\bm\Omega')| \leq \frac{3\alpha^2}{(1-\alpha\varepsilon)^4}R^4(\Delta \bm\Omega',\I;\I).
\end{equation*}
\normalsize
Use formulas (\ref{sec_r_mom}), (\ref{t_r_mom}) to get
\small
\begin{align*}
&H_{\overline{\bm\Omega'}}(\Delta\bm\Omega') = \norm{\Delta\bm\Omega'}_F^2 - 2\alpha\Tr{\Delta\bm\Omega'^3}  - \frac{\(\Tr{\Delta\bm\Omega'}\)^2+2\norm{\Delta\bm\Omega'}_F^2}{p+2} \\ &+2\alpha\frac{(\Tr{\Delta\bm\Omega'})^3+6\Tr{\Delta\bm\Omega'}
\norm{\Delta\bm\Omega'}_F^2+8\Tr{\Delta\bm\Omega'^3}}{(p+2)(p+4)}\\
&+ r(\alpha\Delta\bm\Omega')\norm{\Delta\bm\Omega'}_F^2\\
&=H_{\bm\Omega_0}(\Delta\bm\Omega) + r(\alpha\Delta\bm\Omega')\norm{\Delta\bm\Omega'}_F^2 \\
&-2\alpha\(\frac{(p^2+6p)\Tr{\Delta\bm\Omega'^3} - (\Tr{\Delta\bm\Omega'})^3-6\Tr{\Delta\bm\Omega'}\norm{\Delta\bm\Omega'}_F^2}{(p+2)(p+4)}\)
\end{align*}
\begin{align*}
&\geq \(\frac{p}{p+2}\cos^2{\phi} - 2\alpha\norm{\Delta\bm\Omega'}_2 + r(\alpha\Delta\bm\Omega')\)\norm{\Delta\bm\Omega'}_F^2,
\end{align*}
\normalsize
where due to the condition $\norm{\Delta\bm\Omega'}_2 = \varepsilon \leq \frac{1}{6}$,
\begin{align*}
&|r(\alpha\Delta\bm\Omega')| \leq \frac{3(\alpha\varepsilon)^2}{(1-\alpha\varepsilon)^4} + p\frac{3\alpha^2}{(1-\alpha\varepsilon)^4}\frac{2\norm{\Delta\bm\Omega'}_F^2}{p^2} \\
&\leq \frac{(\alpha\varepsilon)^2(3+\frac{6}{p}p\varepsilon^2)}{(1-\alpha\varepsilon)^4} \leq \frac{\varepsilon^2(3+6\varepsilon^2)}{(1-\varepsilon)^4} \leq \frac{1}{6}.
\end{align*}
Here we have applied the bound from (\ref{r_k_bound}) to $R^4(\Delta \bm\Omega',\I;\I)$.
Finally,
\begin{equation*}
2\alpha\norm{\Delta\bm\Omega'}_2 \leq 2\varepsilon \leq \frac{1}{3}
\end{equation*}
\begin{equation*}
H_{\overline{\bm\Omega'}}(\Delta\bm\Omega') \geq \frac{p}{2(p+2)}\cos^2{\phi_0}\norm{\Delta\bm\Omega'}_F^2.
\end{equation*}
\end{proof}

\begin{proof}[Proof of Corollary \ref{cor_lem_3}]
Set
\begin{equation*}
\frac{1}{2\sqrt{2}}\frac{1+1/p^2}{\ln(32\sqrt{2}p^2)} \leq \eta = \tau\frac{(1-\varepsilon)^2}{\sqrt{2}}\frac{p}{2(p+2)}\cos^2{\phi_0} \leq \frac{1}{2\sqrt{2}},
\end{equation*} recall that $\varepsilon \leq \frac{p}{6(p+2)}\cos^2{\phi_0} \leq \frac{1}{6}$ and use Lemma \ref{lem:lem_2} to get
\begin{align*}
&\mathbb{P}\(\nabla^2 \widehat{f_{\overline{\bm\Omega'}}}(\U') \leq (1-\tau)\frac{p}{2(p+2)}\cos^2{\phi_0}\norm{\U'}_F^2\) \\
&\leq 2p^2\exp\(-\frac{n\eta\sqrt{2}}{8\ln(32\sqrt{2}p^2)}\)\(1+\frac{7}{n^2\eta^4}\) \\
&\leq 2p^2\exp\(-\frac{n\tau\cos^2{\phi_0}}{46\ln(7p)(1+\frac{2}{p})}\)\(1+\frac{2\cdot10^3(1+\frac{2}{p})^4}{n^2\tau^4\cos^8{\phi_0}}\).
\end{align*}
\end{proof}

\end{appendices}

\bibliographystyle{IEEEtran}
\bibliography{ilya_bib}

\end{document}